\documentclass[a4paper,11pt]{article}

\usepackage[T1]{fontenc}
\usepackage[utf8]{inputenc}
\usepackage[DIV=12]{typearea}

\usepackage{amsmath}
\usepackage{amsfonts}
\usepackage{amssymb}
\usepackage{amsthm}
\usepackage{mathtools}
\usepackage{bigints}

\usepackage{csquotes}

\usepackage{braket}
\usepackage{enumerate} 
\usepackage{cancel}
\usepackage{dsfont} % mathds (doublestroke)

\usepackage[sc]{mathpazo}
\usepackage[backend=biber,
            language=english,
            style=alphabetic, %numeric-comp,%
            firstinits=true,
            maxbibnames=5, % default: 3, et al.
            doi=false,
            isbn=false,
            url=false
           ]{biblatex}

\AtEveryBibitem{%
  \clearfield{pages}%
  \clearfield{month}%
}

\usepackage{hyperref}

\hypersetup{%
    colorlinks=false, hidelinks, bookmarksnumbered,pdfencoding=auto,
    bookmarksopen=true, bookmarksopenlevel=1, unicode,
    pdftitle={Dynamic Inverse Wave Problems -- Part I:\@ Regularity for the Direct Problem}, %chktex 8
    pdfsubject={Dynamic Inverse Wave Problems -- Part I:\@ Regularity for the Direct Problem}, %chktex 8
    pdfauthor={Thies Gerken, Simon Grützner},%
    pdfkeywords={Regularity, Hyperbolic, Evolution Equations, Time-dependent, PDE}
}

\usepackage{aliascnt}
\def\NewTheorem#1[#2]#3{%
  \newaliascnt{#1}{#2}
  \newtheorem{#1}[#1]{#3}
  \aliascntresetthe{#1}
  \expandafter\def\csname #1autorefname\endcsname{#3} %chktex 41
}

\newtheorem{theorem}{Theorem}[section]
\NewTheorem{corollary}[theorem]{Corollary}
\NewTheorem{lemma}[theorem]{Lemma}
\NewTheorem{proposition}[theorem]{Proposition}

\newtheorem*{theorem:main}{Theorem~\ref{theorem:main}}

\theoremstyle{definition}

\newtheorem{example}[theorem]{Example}

\theoremstyle{remark}
\newtheorem{remark}[theorem]{Remark}

\numberwithin{equation}{section}

\DeclareMathOperator*{\esssup}{ess\,sup}

\DeclareMathOperator{\divv}{div}
\DeclareMathOperator*{\compose}{\bigcirc}
\DeclareMathOperator*{\spann}{span}

\newcommand\R{\ensuremath{\mathbb R}}
\newcommand\N{\ensuremath{\mathbb N}}

\newcommand\map[3]{\ensuremath{{#1}\,\colon \,{#2}\to{#3}}}

\newcommand{\subnorm}[2]{{{\left \|#2\right \|}_{#1}}}
\newcommand{\norm}[1]{{\subnorm{{}}{#1}}}
\newcommand{\tnorm}[1]{{\|#1 \|}}

\newcommand\dt{\ensuremath{{\,\mathrm{d}t}}}
\newcommand\ds{\ensuremath{{\,\mathrm{d}s}}}
\newcommand\dx{\ensuremath{{\,\mathrm{d}x}}}

\newcommand\dd{\ensuremath{{\,\mathrm{d}}}}
\newcommand\ddt{\ensuremath{{\frac{\dd}{\dt}\,}}}

\newcommand\scp[2]{{\left({#1},{#2} \right)}}
\newcommand\dup[2]{{\langle{#1},{#2} \rangle}}

\newcommand{\assign}{\coloneqq}

\renewcommand{\phi}{\varphi}
\renewcommand{\epsilon}{\varepsilon}
\newcommand{\eps}{\epsilon}

\title{Dynamic Inverse Wave Problems -- Part I:\\ Regularity for the Direct Problem} %chktex 8

\author{Thies Gerken\thanks{Center for Industrial Mathematics, Universit\"at Bremen, Germany; \texttt{tgerken@math.uni-bremen.de}} \and Simon Grützner\thanks{Center for Industrial Mathematics, Universit\"at Bremen, Germany; \texttt{sg@math.uni-bremen.de}}}

\date{\today}

\begin{document}

\maketitle

\begin{abstract}
  For parameter identification problems the Fréchet-derivative of the parameter-to-state map is of particular interest.
  In many applications, e.g.\ in seismic tomography, the unknown quantity is modeled as a coefficient in a linear differential equation, therefore computing the derivative of this map involves solving the same equation, but with a different right-hand side.
  It then remains to show that this right-hand side is regular enough to ensure the existence of a solution. 
  For second-order hyperbolic PDEs with time-dependent parameters the needed results are not as readily available as in the stationary case, especially when working in a variational framework. 
  This complicates for example the reconstruction of a time-dependent density in the wave equation. 
  To overcome this problem we extend the existing regularity results to the time-dependent case.
\end{abstract}

% \tableofcontents

\section{Introduction}

One of the motivations of this work is the identification of the space- and time-dependent mass density $\rho$ from the solution $u$ of the wave equation
\begin{subequations}\label{eq:intro_waveeq}
\begin{gather}
  u'' - \divv \frac{\nabla u}{\rho} = f \ \text{ in $[0,T]\times \Omega$}, \\
  u = 0 \ \text{ in $[0,T]\times \partial \Omega$},\ \ u(0)=u_0\ \text{and}\ u'(0)=u_1.
\end{gather}
\end{subequations}
A common approach is to write this problem as an abstract evolution equation of the form
\begin{subequations}\label{eq:intro_problem}
  \begin{gather}\label{eq:intro_problem_pde}
    u''(t)+A(t)u(t) = f(t)\text{ in } V^*, \text{ a.e.\ in } I,\\ \label{eq:1.1b}
    u(0)=u_0\text{ in } H,\ u'(0)=u_1\text{ in } V^*,
  \end{gather}
\end{subequations}
where $I:=(0,T)$ with $0<T<\infty$, $V=H^1_0(\Omega)$, $H=L^2(\Omega)$ and $A(t)\in \mathcal L(V; V^*)$ is given via
\begin{equation*}
  \dup{A(t)\phi}\psi_{V^*\times V} = \int_\Omega \frac{\nabla\phi \cdot \nabla \psi}{\rho(t)} \dx, \quad \phi,\psi\in V.
\end{equation*}
This abstract problem is then analyzed using semigroup theory~\cite{pazy83} or variational techniques~\cite{zeidler1985, lionsmagenes:bvp}.
Either of these tools yield a well-defined solution $u\in L^2(I; V)\cap H^1(I; H)\cap H^2(I; V^*)$ for a strongly positive $A\in W^{1,\infty}(I; \mathcal L(V; V^*))$, i.e. $\rho \in W^{1,\infty}(I; L^\infty(\Omega))$ with $\rho(t,x) \geq \rho_0 > 0$ for almost all $(t,x)\in I\times \Omega$.
Since our task is to reconstruct $\rho$ from $u$, the parameter-to-state map reads $F:\rho \mapsto u$ and in order to apply Newton-based algorithms for the inversion of $F$, knowledge of its Fréchet derivative is crucial.
Formal derivation of~\eqref{eq:1.1b} shows that if this derivative exists, then $u_h = \partial F(\rho)[h]$ with pertubation $h\in W^{1,\infty}(I; L^\infty(\Omega))$ would have to solve
\begin{equation*}
  u_h''(t)-\divv \frac{\nabla u_h(t)}{\rho(t)} = -\divv \left(h(t) \frac{\nabla u(t)}{\rho(t)^2}\right)
\end{equation*}
in a weak sense and satisfy homogeneous initial conditions.
Without further information about $u$ we only know that the right-hand side of this equation lies in $L^2(I; V^*)$, which (in contrast to elliptic or parabolic equations, see e.g.~\cite{GiTr83,LaEv10,hallerrehberg_maxpara08,arendt07} and the references therein) is not enough to ensure existence of $u_h$.
The equation would be solvable if the right-hand side was an element of $H^1(I; V^*)$, which requires $u\in H^1(I; V)$.
This is the reason why we are interested in obtaining higher regularity of the solution $u$.\\[2ex]
A prominent approach to obtain higher regularity is the formal differentiation of the abstract formulation~\eqref{eq:intro_problem_pde} with respect to time, which yields
\begin{equation*}
  u^{\prime\prime\prime}(t) + A'(t)u(t)+A(t)u'(t) = f'(t),
\end{equation*}
and then treating this (complemented by suitable initial values) as a new problem for $v\assign u'$, c.f. \textcite{JoWl87}. 
The expression $A'(t)u(t)$ is regarded as independent of $v$ and moved to the right-hand side. 
The resulting equation then reads $v''(t)+A(t)v(t) = f'(t)-A'(t)u(t)$. 
Now, to ensure the existence of such a $v$, one needs that $t\mapsto A'(t)u(t)$ is an element of either $L^2(I; H)$ or $H^1(I; V^*)$. 
The latter would require $u\in H^1(I; V)$, which is what we are trying to show in the first place. 
The former can be fulfilled by assuming $A'(t)\in \mathcal L(V; H)$, which is clearly violated in our example due to the time dependence of $\rho$.
Another possible approach is to first show spatial regularity for $u$ and then use integration by parts to see that $A'(t)u(t)$ can indeed be applied to elements of $H$, i.e. $u(t)\in \mathcal D(A'(t))$.
However, this is not possible without using specific knowledge about the Hilbert spaces $H$ and $V$.\\[2ex]
We overcome all of these problems by regarding $u$ as dependent on $v$ by writing $u(t)=u_0+\int_0^t v(s)\ds$. 
This results in a mixed integral and differential equation for $v$, which we analyze using common variational techniques. 
Afterwards we show that $v\in L^2(I; V)\cap H^1(I; H)$ indeed equals $u^\prime$ and therefore $u\in H^1(I; V) \cap H^2(I; H)$. 
By iterating this process we can get even higher regularity for $u$. \\[2ex]
In principle one could tackle this kind of abstract problem using semigroups. 
However, to our knowledge, there are only two similar result in literature. 
One is stated by \textcite{lions_equations_1961} where he uses a variational approach and transforms the time interval from $(0,T)$ to $(0,\infty)$ and then solves the weak formulation in the space-time domain using a weaker form of Lax-Milgram. 
The resulting assumptions regarding the smoothness of the operators are the same, but his theory only applies to homogeneous initial conditions and requires that the first $k-1$ derivatives of the right-hand side vanish at zero. \\[2ex]
The other one is from~\cite{ToKa85}, where the author treats the abstract problem from~\eqref{eq:intro_eq_general} for $C$ set to the identity in $\mathcal{L}(H)$ and to $B,Q=0$. 
The needed compatibility condition agree with ours in this specific setting. 
The author also introduces some additional structure to $V$ and $H$ which then makes some additional abstract spatial regularity result possible.\\[2ex]
The analysis is done on the abstract problem~\eqref{eq:intro_problem} and is therefore not limited to the wave equation. 
It can directly be applied to related phenomena like the elastic wave- or Maxwell's equations.
Following this notion we want to include even more general cases of multiple unknown parameters, or a parameter that appears in more than one coefficient of the PDE.\@
Hence we augment~\eqref{eq:intro_problem} even further by additional time-dependent operators $B, C, Q$ and try to achieve higher differentiability with respect to time of the solution of
\begin{equation}\label{eq:intro_eq_general}
  \ddt (C(t)u'(t)) + B(t)u'(t) + (A(t)+Q(t))u(t) = f(t)\text{ in } V^*, \text{ a.e.\ in } I.
\end{equation}
The operator $Q(t)\in \mathcal L(V; H)$ may include asymmetric parts in the PDE that depend only on first-order spatial derivatives, e.g.\ a transport term $\nabla u(t) \cdot b(t)$.
The operators $B,C$ also generate time-dependent zero order terms after differentiating this equation more than once. 
In contrast to $A$ these expressions can safely be moved to the right-hand side, but we aim to get a single regularity theorem that is applicable to a general equation like~\eqref{eq:intro_eq_general}.
An example for a parameter appearing in multiple operators would be the equation $u''/\rho - \divv(\nabla u / \rho) = f$. 
This is a more realistic version of our introductory PDE, is also used in~\cite{KiRie13} and can be formulated as~\eqref{eq:intro_eq_general} as long as we write the left-hand side as $(u'/\rho)' + (\rho' / \rho^2) u'  - \divv(\nabla u / \rho)$.\\[2ex]
Many articles on parameter identification in hyperbolic PDEs focus on equations with time-independent principal part, e.g.~\cite{KiRie13},~\cite{LeSchl17},~\cite{blazek_mathematical_2013} and~\cite{KiRie16}. 
Our results give the means to treat these problems in a setting where the parameters are time-dependent. 
Thus it is now possible to tackle the corresponding dynamic inverse problems, such as waves propagating with a time-dependent wave speed or density.
These results also benefit some of today's other applications like optimization, see e.g.~\cite{FrTr10} or~\cite{NeiTi94}, or solving quasi-linear PDEs with Newton-like methods, where knowledge of the Fréchet derivative is important as well.\\[2ex]
This article is organized as follows: We begin by stating an existence and uniqueness result for equation~\eqref{eq:intro_eq_general} and our regularity theorem, which we then aim to prove. 
For this, we need solvability of an auxiliary problem, which we analyze in \autoref{section:auxiliary}. 
This then allows for a compact proof of our main theorem in \autoref{section:main_proof}.

\section{Preliminaries and Main Theorem}

Let $V$, $H$ denote real Hilbert spaces where $V$ is separable and the embedding $V\hookrightarrow H$ is dense, so that $V\hookrightarrow H \hookrightarrow V^*$ establishes a \emph{Gelfand triple}. 
The dual space of $V$ is denoted by $V^*$. 
Without loss of generality we assume that $\norm{\cdot}_H \leq \norm{\cdot}_V$. 
For a Banach space $X$ we denote by $W^{k,p}(I; X)$ the usual Bochner space of $X$-valued time-dependent functions (see e.g.~\cite{zeidler1985}).
Furthermore, $\mathcal L(X;Y)$ is the space of linear and continuous operators from $X$ to $Y$, and in the case $X=Y$ we just write $\mathcal L(X)$. 
The inner product of $H$ is denoted by $\scp{\cdot}{\cdot}$, and the duality pairing between $V^*$ and $V$ by $\dup{\cdot}{\cdot}$.

Taking~\eqref{eq:intro_eq_general} complemented by suitable initial conditions then reads
\begin{subequations}\label{eq:problem_general}
  \begin{gather}\label{eq:problem_general_a}
    \ddt (C(t)u'(t))+B(t)u'(t)+(A(t)+Q(t)) u(t) = f(t)\text{ in } V^*, \text{ a.e.\ in } I,\\ \label{eq:problem_general_b}
    u(0)=u_0\text{ in } H,\ (C(\cdot) u'(\cdot))(0)=C(0)u_1\text{ in } V^*,
  \end{gather}
\end{subequations}
where $A\in L^\infty(I; \mathcal L(V; V^*))$, $B,C\in L^\infty(I; \mathcal L(H))$ and $Q\in L^\infty(I; \mathcal L(V; H))$.
The operators $A$ and $C$ are assumed to be pointwise self-adjoint and strongly positive, i.e. $\scp{C(t)\phi}{\phi}\geq c_0 \norm{\phi}^2_H$ and $\dup{A(t)\phi}{\phi} \geq a_0 \norm{\phi}^2_V$ with $a_0>0$ and $c_0 > 0$.
In the case that $A$ only fullfills the weaker Gårding-inequality $\dup{A(t)\phi}{\phi} \geq a_0 \norm{\phi}^2_V - \lambda \norm{\phi}^2_H$ with $\lambda\in\R$ one may remedy this by replacing $A$ with $A+\lambda I$ and $Q$ with $Q-\lambda I$.
For notational simplicity we define the operator
\begin{equation*}
  \mathcal C: L^2(I; H)\to L^2(I; H), \quad (\mathcal Cu)(t) = C(t)u(t),
\end{equation*}
and analogously $\mathcal A$, $\mathcal B$ and $\mathcal Q$. 
This allows us to write $(\mathcal Cu')'(t)$ instead of $(C(\cdot) u'(\cdot))'(t)$ and equation~\eqref{eq:problem_general_a} becomes $(\mathcal Cu')'+\mathcal Bu' + (\mathcal A + \mathcal Q) u = f$ (in the $L^2(I; V^*)$ sense). 
We start by stating suitable conditions for the solvability of this problem.

\begin{lemma}[Well-posedness]\label{theorem:well_posedness}
  Provided that $A\in W^{1,\infty}(I; \mathcal L(V;V^*))$, $B,C\in W^{1,\infty}(I; \mathcal L(H))$, $Q\in W^{1,\infty}(I; \mathcal L(V; H))$, $f\in L^2(I; H)$ or $f\in H^1(I; V^*)$, $u_0\in V$, and $u_1\in H$
  there exists a uniquely determined $u\in L^2(I; V)\cap H^1(I; H)$ with $(\mathcal Cu')' \in L^2(I; V^*)$ solving~\eqref{eq:problem_general}.
\end{lemma}

\begin{proof}
  The proof for the case where the operators are continuously differentiable in time can be found in~\cite{DauLi00E}, but the assertion also holds for $W^{1,\infty}$ because all assumptions on the operators only have to be satisfied up to a set of measure zero. 
  The existence of a solution can also be deduced from the following \autoref{corollary:aux_solvable} by setting $k=0$ in the corresponding proof, which is correct without any differentiability assumptions on $B$ and $Q$. 
  They arise when showing uniqueness of the solution. 
  If $Q\in L^\infty(I; \mathcal L(H))$ then this regularity of $Q$ would not be needed.
\end{proof}

We will make frequent use of the following lemma, which can be seen as a generalization of the product rule.

\begin{lemma}\label{lemma:product rule}
  Let $X, Y$ be separable Banach spaces and $G\in W^{1,\infty}(I; \mathcal L(X; Y^*))$. 
  Given any $u\in H^1(I; X)$, $v\in H^1(I; Y)$ the map $t\mapsto \dup{G(t)u(t)}{v(t)}$ belongs to $W^{1,1}(I)$ and we have
  \begin{equation}\label{eq:productrule2}
    \ddt \dup{G(t)u(t)}{v(t)}_{Y^*\times Y} = \dup{G^\prime(t)u(t)}{v(t)}_{Y^*\times Y} + \dup{G(t)u^\prime(t)}{v(t)}_{Y^*\times Y} + \dup{G(t)u(t)}{v^\prime(t)}_{Y^*\times Y},
  \end{equation}
  which holds for almost all $t\in I$.
\end{lemma}

\begin{proof}
  We only need to show that the assertion holds for $G\in C^1(I; \mathcal L(X; Y^*))$.
  It is clear that both the right-hand side of~\eqref{eq:productrule2} and the mapping $t\mapsto \dup{G(t)u(t)}{v(t)}_{Y^*\times Y}$ belong to $L^1(I)$.
  Due to dense embeddings there exist $u^\eps\in C^\infty(I; X)$, $v^\eps\in C^\infty(I; Y)$ such that $u^\eps \to u$ in $H^1(I; X)$ and $v^\eps\to v$ in $H^1(I; Y)$ when $\epsilon\to 0$.
  Due to the chain rule we have
  \begin{equation*}
    \ddt \dup{G(t)u^\eps(t)}{v^\eps(t)} = \dup{G^\prime(t)u^\eps(t)}{v^\eps(t)} + \dup{G(t){(u^\eps)}^\prime(t)}{v^\eps(t)} + \dup{G(t)u^\eps(t)}{{(v^\eps)}^\prime(t)}
  \end{equation*}
  in the classical sense. 
  In particular,
  \begin{align}
    -\int_0^T &\phi^\prime(t) \dup{G(t)u^\eps(t)}{v^\eps(t)} \dt \notag \\
    &= \int_0^T \phi(t) \left[\dup{G^\prime(t)u^\eps(t)}{v^\eps(t)} + \dup{G(t){(u^\eps)}^\prime(t)}{v^\eps(t)} + \dup{G(t)u^\eps(t)}{{(v^\eps)}^\prime(t)}\right]\dt. \label{eq:productrule1}
  \end{align}
  Both sides of this equation converge to the respective terms evaluated at $u$ and $v$ when $\eps\to0$.
\end{proof}

We follow the usual technique of showing that derivatives of $u$ are themselves solutions of evolution equations.
Formally differentiating~\eqref{eq:problem_general} $k$ times with respect to $t$ leads to
\begin{subequations}
\begin{align}
    f^{(k)} &= (\mathcal C u^{(k+1)})'+ (k\mathcal C'+\mathcal B)u^{(k+1)}+ \left(\mathcal A + \mathcal Q + k\mathcal B' + \frac{k(k+1)}2 \mathcal C''\right) u^{(k)} \notag \\
    &\ \ + \sum_{j=1}^{k} \left[{k \choose j} (\mathcal A^{(j)} + \mathcal Q^{(j)})  + {k \choose j+1} \mathcal B^{(j+1)} + {k+1 \choose j+2}\mathcal C^{(j+2)}\right] u^{(k-j)}, \label{eq:problem_general_aux_u_eq}
\end{align}
in the sense of $L^2(I; V^*)$, together with initial values
\begin{equation}
u^{(k)}(0)=u_k\text{ in } H,\ (\mathcal Cu^{(k+1)})(0)=C(0)u_{k+1}\text{ in } V^*,
\end{equation}
\end{subequations}
which are given recursively for $k\geq 0$ through
\begin{align}
  C(0)u_{k+2} &= f^{(k)}(0) - ((k+1) C'(0) + B(0)) u_{k+1} \notag \\
  &\ \ - \sum_{j=0}^{k} \left[ {k \choose j} (A^{(j)}(0) + Q^{(j)}(0)) + {k \choose j+1} B^{(j+1)}(0) +  {k+1 \choose j+2} C^{(j+2)}(0)\right] u_{k-j},  \label{eq:def_initial_values}
\end{align}
starting from already given $u_0$ and $u_1$. 
Note that these compatibility conditions resemble those from Theorem 30.1 in~\cite{JoWl87} if we were to choose $C$ to be the identity in $\mathcal{L}(H)$, $B,Q=0$, and $A$ as time-independent. \\[2ex]
Our intention is to interpret~\eqref{eq:problem_general_aux_u_eq} as a second-order evolution equation for $u^{(k)}$ without moving lower order derivatives of $u$ to the right-hand side. 
To this end, for some Banach space $X$ and $g\in X$ we introduce the operator $R_{X,g}\colon L^2(I; X) \to H^{1}(I; X)$ via
\begin{equation}\label{eq:2.4}
  (R_{X,g}v)(t):=g+\int_0^t v(s)\,\ds.
\end{equation}
In view of an efficient notation, we abbriviate the operator $R_{V\!,g}$ by $R_g$ if not explicitly stated otherwise. 
We also agree on the definition and notation of the following operator, which denotes the consecutive composition of $R_{X,u_l}$ for $l=m,\dots,n$, i.e.,
\begin{equation}\label{eq:2.5}
  \compose_{l=m}^{n}R_{X,u_l}v :=\left \{\begin{array}{rl} (R_{X,u_m}\circ\cdots\circ R_{X,u_n})v, & m\le n \\ v, & \textnormal{else.} \end{array}\right. %chktex 26
\end{equation}
Writing $v\assign u^{(k)}$ now leads to the auxiliary problem
\begin{subequations}\label{eq:problem_general_aux}
\begin{align}\label{eq:problem_general_aux_eq}
    f^{(k)} &= (\mathcal C v')'+ (k\mathcal C'+\mathcal B)v'+(\mathcal A+\mathcal Q)v  + \left( k\mathcal B' + \frac{k(k+1)}2 \mathcal C''\right) v \\
    &\ \ + \sum_{j=1}^{k} \left[{k \choose j} (\mathcal A^{(j)}+\mathcal Q^{(j)}) + {k \choose j+1} \mathcal B^{(j+1)} + {k+1 \choose j+2}\mathcal C^{(j+2)}\right] \compose_{l=k-j}^{k-1} R_{u_l} v, \notag
\end{align}
equipped with initial values
\begin{equation}\label{eq:problem_general_aux_ic}
v(0)=u_k\text{ in } H,\ (\mathcal Cv')(0)=C(0)u_{k+1}\text{ in } V^*.
\end{equation}
\end{subequations}
Note that the left-hand side of~\eqref{eq:problem_general_aux_eq} is nonlinear in $v$ because $R_g$ is affine linear for $g\neq 0$.
Also, the dependence of~\eqref{eq:problem_general_aux_eq} and~\eqref{eq:def_initial_values} on $C^{(k+2)}$ and $B^{(k+1)}$ is merely notational, because the coefficients in front of them vanish.
It is also worth mentioning that we explicitly wrote the zeroth index of the sum in~\eqref{eq:problem_general_aux_eq} since we need strong positivity properties from $\mathcal A$ and $\mathcal C$. \\[2ex]
The main assertion of this article is the following regularity result, which we are going to prove by showing that solutions to~\eqref{eq:problem_general_aux} exist and that they have to be equal to $u^{(k)}$.

\begin{theorem}\label{theorem:main}
  Let $k\in\N\cup \{0\}$, $l=\max \{1,k\}$ and suppose that $A\in W^{k+1,\infty}(I; \mathcal L(V;V^*))$, $Q\in W^{l,\infty}(I; \mathcal L(V;H))$, $B\in W^{l,\infty}(I; \mathcal L(H))$, $C\in W^{k+1,\infty}(I; \mathcal L(H))$, as well as
  $f\in H^k(I; H)$ or $f\in H^{k+1}(I; V^*)$, and $u_j\in V$ for $j=0,\dots, k$, $u_{k+1}\in H$.
  Then the unique solution $u$ of problem~\eqref{eq:problem_general} lies in $H^k(I; V)\cap H^{k+1}(I; H)$ with $(\mathcal C u^{(k+1)})' \in L^2(I; V^*)$ and satisfies the energy estimate
  \begin{equation}\label{eq:theorem_main_energy}
  \norm{u}_{W^{k,\infty}(I; V)}^2 + \tnorm{u^{(k+1)}}_{L^{\infty}(I; H)}^2 \leq \Lambda \left(\sum_{j=0}^k \norm{u_j}_V^2 + \norm{u_{k+1}}_H^2 + \norm{f}^2 \right),
  \end{equation}
  where $f$ is measured in either the $H^k(I; H)$- or the $H^{k+1}(I; V^*)$ norm and $\Lambda=\Lambda(k)$ is a constant depending continuously on $1/{c_0}$, $1/{a_0}$, $T$ and the operators $A, B, C, Q$, measured in the spaces above.
\end{theorem}

We would like to give an example for the application of this regularity result to a partial differential equation. 
In particular, we are interested in showing what the conditions for the $u_i$ entail in practice. 
For simplicity we use the wave equation with a time- and space-dependent coefficient in the divergence term. %, but the example can easily be extended to general hyperbolic PDEs.

\begin{example}
Suppose $\Omega\subset \R^n$ is a bounded domain.
The problem
\begin{align*}
  u''(t,x) - \divv(a(t,x) \nabla u(t,x)) &= f(t,x) \text{ for all } (t,x)\in I\times\Omega, \\
  u =0\text{ on }I\times\partial\Omega ,\ u(0,\cdot)&=u_0, \ u'(0, \cdot)=u_1,
\end{align*}
possesses a unique weak solution $u\in L^2(I; H^1_0(\Omega))\cap H^1(I; L^2(\Omega)) \cap H^2(I; H^{-1}(\Omega))$
(according to \autoref{theorem:well_posedness}) if $u_0\in H^1_0(\Omega)$, $u_1\in L^2(\Omega)$, $f\in L^2(I; L^2(\Omega))$ and $a\in W^{1,\infty}(I; L^\infty(\Omega))$,
provided that $a(t,x) \geq a_0 > 0$ almost everywhere in $I\times \Omega$.
Applying~\autoref{theorem:main} to this problem yields $u\in H^k(I; H^1_0(\Omega))\cap H^{k+1}(I; L^2(\Omega)) \cap H^{k+2}(I; H^{-1}(\Omega))$ if $f\in H^k(I; L^2(\Omega))$ and $a\in W^{k+1,\infty}(I; L^\infty(\Omega))$.
Additionally the initial values defined in~\eqref{eq:def_initial_values} have to satisfy $u_i\in H^1_0(\Omega)$ ($i=0, \dots, k$) and $u_{k+1}\in L^2(\Omega)$. 
These can be seen as \emph{compatibility conditions} and can be fulfilled by either one of the following assumptions:
\begin{enumerate}[(i)]
  \item Homogeneous initial data $u_0=u_1=0$ and $f^{(i)}(0) = 0$ for $i=0, \dots, k-1$. 
  In this case the regularity result is the same as obtained by Lions~\cite{lions_equations_1961}.
  % \item $a^{(i)}(0)=0$ for $i=0, \dots, k-1$, $f^{(i)}(0)\in H^1_0(\Omega)$ for $i=0, \dots, k-2$ and $f^{(k-1)}(0)\in L^2(\Omega)$
  % geht nicht, da a(0)=0 nicht erlaubt! -> Immernoch rekursive definition und Anforderungen an $u_0, u_1$
  \item $f^{(k-j)}(0)\in H_0^{j-1}(\Omega)$, $a^{(k-j)}(0)\in W^{j,\infty}(\Omega)$ for $j=1, \dots, k$ and smooth initial values $u_0\in H^{k+1}_0(\Omega)$ and $u_1\in H^k_0(\Omega)$.
  \begin{proof}[Proof (Sketch).]
    Use integration by parts to see that e.g.
    \begin{equation*}
      A(0)g = v\mapsto \int_\Omega a(0,x) \nabla v(x) \nabla g(x) \dx = -\divv (a(0,\cdot) \nabla g)
    \end{equation*}
    (where the last equality is to be understood in the distributional sense) is not only an element of $H^{-1}(\Omega)$, but also lies in $H^k_0(\Omega)$ if $a(0)\in W^{k+1,\infty}(\Omega)$ and $g\in H^{k+2}_0(\Omega)$.
    From this we conclude that $u_k$ lies in $H^l_0(\Omega)$ if $f^{(k-2)}(0)\in H^l_0(\Omega)$, $a^{(j)}(0)\in W^{l+1, \infty}(\Omega)$ and $u_j\in H^{l+2}_0(\Omega)$ for all $j=0, \dots, k-2$.
    This is satisfied for $k\geq 2$ if $f^{(k-j)}(0)\in H^{l+j-2}_0(\Omega)$, $a^{(k-j)}(0)\in W^{l+j-1, \infty}(\Omega)$ $u_j\in H^{l+2}_0(\Omega)$, $u_0\in H^{l+k}_0(\Omega)$
    and $u_1\in H^{l+k-1}_0(\Omega)$ by induction over $k\geq 2$, because the assumptions for $k$ imply the assumptions for all $j=0,\dots, k-2$ with smoothness $l+2$. 
    Therefore $u_j\in H^{l+2}_0(\Omega)$ for $j=0,\dots, k-2$ holds. \\
    The assertion follows by applying this to $u_k\in H^1_0(\Omega)$, $u_{k+1}\in L^2(\Omega)$.
  \end{proof}
\end{enumerate}
\end{example}

\section{Well-posedness of the auxiliary problems}\label{section:auxiliary}

We need solvability of the auxiliary problems~\eqref{eq:problem_general_aux} for $k\in\N$. 
They can be written in the form
\begin{subequations}\label{eq:problem_general_aux_form}
  \begin{gather} \label{eq:problem_general_aux_form_a}
    (\mathcal C v')'+\mathcal Bv'+(\mathcal A+\mathcal Q)v + \sum_{i=1}^k (\mathcal D_i+\mathcal E_i) (R_0^i v) = f\text{ in } L^2(I; V^*), \\ \label{eq:problem_general_aux_form_b}
    v(0)=v_0\text{ in } H,\ (\mathcal Cv')(0)=v_1\text{ in } V^*,
  \end{gather}
\end{subequations}
with different $B$, $Q$, $f$ and suitably defined $D_i, E_i$, where $D_i(t)\in \mathcal L(V; V^*)$ while $E_i(t)\in \mathcal L(V; H)$. 
We analyze this equation and then transfer the results to~\eqref{eq:problem_general_aux}.
To this end, we start by discretizing~\eqref{eq:problem_general_aux_form} in space by taking an orthogonal basis $(\phi_i)_{i\in \N}\subset V$ of $V$ wich is also an orthonormal basis in $H$ and introduce the finite-dimensional subspaces $V_m:=\spann \{\phi_i;\,i=1,\dots,m\}$.
We set $v_m:=\sum_{i=1}^{m}\alpha_i(t)\phi_i$ where $\alpha_i\in W^{2,1}(I)$, so that $v_m\in W^{2,1}(I; V)$. 
We want $v_m$ to solve~\eqref{eq:problem_general_aux_form_a} in $L^2(I; V_m^*)$. 
Testing this equation with $\phi_i$ shows that this is equivalent to
\begin{subequations}\label{eq:3.3}
\begin{multline}\label{eq:3.3a}
  \langle f(t),\varphi_i\rangle =\sum_{j=1}^m\left(\left(\alpha_j'(t)\left( C(t)\varphi_j,\varphi_i\right)\right)'  \right.\\
    + \alpha_j'(t)\left(B(t)\varphi_j,\varphi_i\right) + \alpha_j(t)\left(\langle A(t)\varphi_j,\varphi_i\rangle + (Q(t)\varphi_j,\varphi_i)\right)\\
      \left. + \sum_{l=1}^k (R_{\R,0}^l\alpha_j)(t)\left(\langle D_l(t)\varphi_j,\varphi_i\rangle + (E_l(t)\varphi_j,\varphi_i)\right)\right)
\end{multline}
being satisfied for all $i=1,\dots,m$. 
This $m\times m$-system is complemented by the projected initial values from~\eqref{eq:problem_general_aux_form_b}, that is
\begin{equation}\label{eq:3.3b}
  \scp{v_m(0)}{\phi_i} = \scp{v_0}{\phi_i}, \quad \scp{C(0)v_m'(0)}{\phi_i} = \scp{v_1}{\phi_i},\ \ \text{for } i=1,\dots, m.
\end{equation}
\end{subequations}
These ordinary differential equations are commonly referred to as \emph{Galerkin equations} and their solvability is the subject of the following lemma.

\begin{lemma}[Galerkin solutions]\label{lemma:Galerkin solutions}
  Provided that $A\in W^{1,\infty}(I; \mathcal L(V;V^*))$, $C\in W^{1,\infty}(I; \mathcal L(H))$, $B\in L^\infty(I; \mathcal L(H))$, $Q\in L^\infty(I; \mathcal L(V;H))$, $D_i\in W^{1,\infty}(I; \mathcal L(V;V^*))$, $E_i\in L^\infty(I; \mathcal L(V; H))$ ($i=1, \dots, k$),
    $f\in L^2(I; H)$ or $f\in H^1(I; V^*)$, $v_0\in V$ and $v_1\in H$ the Galerkin equations~\eqref{eq:3.3a},~\eqref{eq:3.3b} admit one unique solution $v_m$.
\end{lemma}
\begin{proof}
  The idea is to transform eq.~\eqref{eq:3.3a} into a system of linear differential equations in the sense of Caratheodory (see~\cite{CoLe55}, Chapter 2, Theorem 1.1, e.g.). 
  To this end, we apply the generalized product rule from \autoref{lemma:product rule} to~\eqref{eq:3.3a} to get
  \begin{multline}\label{eq:3.4}
    \langle f(t),\varphi_i\rangle =\sum_{j=1}^m\bigg(\alpha_j''(t)\left( C(t)\varphi_j,\varphi_i\right)  + \alpha_j'(t)\left((B(t) + C'(t))\varphi_j,\varphi_i\right)  + \alpha_j(t)\left(\langle A(t)\varphi_j,\varphi_i\rangle\right. \\
     + \left(Q(t)\varphi_j,\varphi_i)\right) + \sum_{l=1}^k (R_{\R,0}^l\alpha_j)(t)\left(\langle D_l(t)\varphi_j,\varphi_i\rangle + (E_l(t)\varphi_j,\varphi_i)\right)\bigg)
  \end{multline}
  and then introduce the following matrix-valued functions $\boldsymbol{M}_l \in L^{\infty}(I; \R^{m\times m})$ for $l=-2,\dots,k$ component-wise via
  %
  % \begin{subequations}
  \begin{gather*}
    (\boldsymbol{M}_{-2}(t))_{ij}:=(C(t)\phi_j,\phi_i),\quad (\boldsymbol{M}_{-1}(t))_{ij}:= \left((C'(t)+B(t))\varphi_j,\varphi_i\right), \\
    (\boldsymbol{M}_0(t))_{ij}:= \langle A(t)\varphi_j,\varphi_i\rangle + (Q(t)\varphi_j,\varphi_i), \quad (\boldsymbol{M}_l(t))_{ij}:=\langle D_l(t)\varphi_j,\varphi_i\rangle + (E_l(t)\varphi_j,\varphi_i)
  \end{gather*}
  for $i,j=1,\dots,m$ as well as for $l\ge 2$
  \begin{equation*}
    (\boldsymbol{M}(t))_{l-2+i,l-2+j}:=(\boldsymbol{M}_l(t))_{ij}
  \end{equation*}
  % \end{subequations}
  %
  for $i,j=1,\dots,m$. 
  We also introduce the vector-valued function $\boldsymbol{F}\in W^{1,\infty}(I; \R^{m})$ via $\boldsymbol{F}_i(t):=\langle f(t),\varphi_i\rangle$ for $i=1\dots,m$
  and agree on a new set of variables $\gamma^l:= R_{\R,0}^l\alpha$ for $l=0,\dots,k$, $\gamma=(\gamma^l)_{l=0,\dots,k}$ and $\beta=\alpha'$. 
  This allows rewriting~\eqref{eq:3.4} as a $(k+2)m\times (k+2)m$ system of first-order differential equations, i.e
  \begin{equation*}
    \begin{pmatrix}
      \boldsymbol{M}_{-2}(t) & 0 \\
      0              & \boldsymbol{I}_{(k+1)m}
    \end{pmatrix}
      \begin{pmatrix}
        \beta \\
        \gamma
      \end{pmatrix}'=
        \begin{pmatrix}
          \boldsymbol{M}_{-1}(t) & \boldsymbol{M}(t) \\
          \boldsymbol{I}_m & 0 \\
          0 & \boldsymbol{I}_{km}
        \end{pmatrix}
          \begin{pmatrix}
            \beta \\
            \gamma
          \end{pmatrix}+
            \begin{pmatrix}
              \boldsymbol{F}(t) \\
              0
            \end{pmatrix},
  \end{equation*}
  which is complemented by initial values
  \begin{equation*}
    \boldsymbol{M}_{-2}(0)\beta(0)=\left( \scp{v_1}{\phi_i} \right)_{i=1,\dots, m},\ \ \gamma^0(0) = \left( \scp{v_0}{\phi_i} \right)_{i=1,\dots, m}, \ \ \gamma^l(0)=0 \quad\text{for }l=1,\dots,k.
  \end{equation*}
  Here, $I_m$, $I_{km}$ and $I_{(k+1)m}$ represent identity matrices in $\R^{m\times m}$, $\R^{km \times km}$ and $\R^{(k+1)m\times (k+1)m}$, respectively. 
  Since the $\phi_i$ form an orthogonal basis for $V_m$ and $C(t)$ is invertible due to its coercivitiy, $\boldsymbol{M}_{-2}$ and therefore the whole matrix on the left-hand side are invertible as well.
  It is then straight forward to show the existence of a unique solution $(\beta,\gamma)\in W^{1,\infty}(I; \R^{(k+2)m})$ following~\cite{ORe97} or~\cite{CoLe55} and hence finding a uniquely determined $v_m:=\sum_{j=1}^m\alpha_j(t)\phi_j\in W^{2,\infty}(I; V_m)$ that solves~\eqref{eq:3.3}.
\end{proof}

Our goal is now to establish an upper bound for $v_m$ independent of $m\in\N$.

\begin{lemma}[A priori estimate]\label{lemma:a priori estimate}
    Provided that $A\in W^{1,\infty}(I; \mathcal L(V;V^*)), C\in W^{1,\infty}(I; \mathcal L(H))$, $B\in L^\infty(I; \mathcal{L}(H))$, $Q\in L^\infty(I; \mathcal L(V;H))$, $D_i\in W^{1,\infty}(I; \mathcal L(V;V^*))$, $E_i\in L^\infty(I; \mathcal L(V; H))$ for $i=1, \dots, k$, as well as
    $f\in L^2(I; H)$ or $f\in H^1(I; V^*)$, $v_0\in, V$ and $v_1\in H$ the Galerkin solution $v_m$ satisfies the energy estimate
    \begin{equation}\label{eq:energy}
      \esssup_{t\in I} \left(\norm{v_m(t)}^2_V + \norm{v_m'(t)}^2_H\right) \leq \Lambda \left(\norm{v_0}_V^2 + \norm{v_1}_H^2 + \norm{f}^2 \right),
    \end{equation}
    where $f$ is measured in either the $L^2(I; H)$- or the $H^1(I; V^*)$ norm and $\Lambda>0$ is some constant depending continuously on $1/{c_0}$, $1/{a_0}$, $T$, and the operators $A, B, C, Q, (D_i)_{i=1,\dots, k}, (E_i)_{i=1,\dots, k}$, measured in the spaces above.
\end{lemma}
\begin{proof}
  We take $v_m\in W^{2,\infty}(I; V_m)$ to be the unique solution to the Galerkin equations from the previous lemma, test~\eqref{eq:problem_general_aux_form_a} with $v_m'(t)\in V_m$ and see that
  \begin{align*}
      \scp{(\mathcal C v_m')'(t)}{v_m'(t)} &+ \scp{B(t)v_m'(t)}{v_m'(t)}+ \dup{A(t)v_m(t)}{v_m'(t)} + \scp{Q(t)v_m(t)}{v_m'(t)} \\
      &+ \sum_{i=0}^k \dup{D_i(t) (R_0^i v_m)(t)}{v_m'(t)} + \sum_{i=0}^k \scp{E_i(t) (R_0^i v_m)(t)}{v_m'(t)} = \dup{f(t)}{v_m'(t)}
  \end{align*}
  holds for almost all $t\in I$. 
  We make use of the product rule from \autoref{lemma:product rule} and the fact that $A(t)$ and $C(t)$ are self-adjoint to get
  \begin{align*}
    2\scp{(\mathcal C v_m')'(t)}{v_m'(t)} &= 2\scp{C(t) v_m''(t) + C'(t)v_m'(t)}{v_m'(t)} = \scp{C'(t)v_m'(t)}{v_m'(t)} + \Big( \scp{C(t) v_m'(t)}{v_m'(t)}\Big)', \\
    2\dup{A(t) v_m(t)}{v_m'(t)} &= \Big(\dup{A(t) v_m(t)}{v_m(t)}\Big)' - \dup{A'(t)v_m(t)}{v_m(t)}.
  \end{align*}
  Together with $\norm{v_m(0)}_V \leq \norm {v_0}_V$, $\norm{C(0)v_m'(0)}_H \leq \norm{v_1}_H$ this gives
  \begin{align}
    % W 1,\infty statt L\infty für A weil das an einer bestimmten Stelle (t=0) ausgewertet wird
    a_0 &\norm{v_m(t)}^2_V + c_0 \norm{v_m'(t)}^2_H \notag \\
    &\leq \tfrac 1{c_0} \norm{v_1}_H^2 + \norm{A}_{W^{1,\infty}(I; \mathcal L(V;V^*))} \norm{v_0}_V^2 \label{eq:energy_general1} + \bigintsss_0^t \Big( \scp{C(s) v_m'(s)}{v_m'(s)} + \dup{A(s) v_m(s)}{v_m(s)}\Big)' \ds \notag \displaybreak[0]\\
    &= \tfrac 1{c_0} \norm{v_1}_H^2 + \norm{A}_{W^{1,\infty}(I; \mathcal L(V;V^*))} \norm{v_0}_V^2 + 2\bigintsss_0^t \scp{(\mathcal C v_m')'(s)}{v_m'(s)}  + \dup{A(s) v_m(s)}{v_m'(s)} \ds \notag \\
    & \ \quad -  \bigintsss_0^t \scp{C'(s)v_m'(s)}{v_m'(s)} + \dup{A'(s)v_m(s)}{v_m(s)} \ds \notag \displaybreak[0]\\
    &\leq \tfrac 1{c_0} \norm{v_1}_H^2 + \norm{A}_{W^{1,\infty}(I; \mathcal L(V;V^*))} \norm{v_0}_V^2 + 2\bigintsss_0^t \dup{f(s)}{v_m'(s)} - \scp{Q(s)v_m(s)}{v_m'(s)}  \ds \notag \displaybreak[0] \\
    & \ \quad - 2\bigintsss_0^t \scp{B(s)v_m'(s)}{v_m'(s)} + \sum_{i=0}^k \dup{D_i(s) (R_0^i v_m)(s)}{v_m'(s)} + \dup{E_i(s) (R_0^i v_m)(s)}{v_m'(s)} \ds \notag \\
    & \ \quad - \bigintsss_0^t \scp{C'(s)v_m'(s)}{v_m'(s)} + \dup{A'(s)v_m(s)}{v_m(s)} \ds.
  \end{align}
  Since we are looking to apply Grönwall's inequality, we have to modify the right-hand side so that it only depends on $\norm{v_m(s)}_V^2$ and $\norm{v_m'(s)}^2_H$. 
  We perform this individually for each summand.\\[2ex]
  If $f\in L^2(I; H)$ then $2\dup{f(s)}{v_m'(s)} \leq \norm{f(s)}^2_H + \norm{v_m'(s)}^2_H$. 
  Otherwise, i.e. $f\in H^1(I; V^*)$, we have
  \begin{align*}
    2 \int_0^t \dup{f(s)}{v_m'(s)} \ds &= 2 \dup{f(t)}{v_m(t)}  -  2 \dup{f(0)}{v_m(0)} - 2\int_0^t \dup{f'(s)}{v_m(s)} \ds \\
    &\leq  \epsilon \norm{v_m(t)}^2_V  + \tfrac {1+\epsilon}\epsilon \norm{f}_{H^1(I; V^*)}^2 + \norm{v_0}_V^2 + \int_0^t \norm{f'(s)}_{V^*}^2 + \norm{v_m(s)}^2_V \ds,
  \end{align*}
  which holds for all $\epsilon >0$.\\[2ex]
  For the last summand on the right-hand side we immediately see that
  \begin{multline}
    -\bigintsss_0^t \scp{C'(s)v_m'(s)}{v_m'(s)} + \dup{A'(s)v_m(s)}{v_m(s)} \ds \\
      \leq \norm{C'} \bigintsss_0^t \norm{v_m'(s)}^2_H \ds + \norm{A'} \bigintsss_0^t \norm{v_m(s)}^2_V \ds,
  \end{multline}
  the expression with $B$ can be handled analogously.\\[2ex]
  For terms including $D_i$, we have to use integration by parts because in contrast to parabolic equations we cannot expect $\norm{v_m'(s)}_V$ to be bounded during the limit process $m\to\infty$, which is already apparent on the left-hand side of~\eqref{eq:energy_general1}. 
  In employing~\autoref{lemma:product rule} we see that
  \begin{align*}
    -\bigintsss_0^t &\dup{D_i(s)(R^i_0 v_m)(s)}{v_m'(s)} \ds = -\dup{D_i(t)(R^i_0 v_m)(t)}{v_m(t)} \\
    & \qquad \quad + \bigintsss_0^t \dup{D_i'(s)(R^i_0 v_m)(s)}{v_m(s)} - \dup{D_i(s)(R^{i-1}_0 v_m)(s)}{v_m(s)} \ds \displaybreak[0]\\
    &\leq \norm{D_i}_{L^\infty(I; \mathcal L(V; V^*))} \tnorm{(R^i_0 v_m)(t)}_V \norm{v_m(t)}_V \\
    & \qquad \quad + \bigintsss_0^t \norm{D_i'} \tnorm{(R^i_0 v_m)(s)}_V \norm{v_m(s)}_V + \norm{D_i'} \tnorm{(R^{i-1}_0 v_m)(s)}_V \norm{v_m(s)}_V \ds \displaybreak[0]\\
    &\leq T^{i-1} \norm{D_i}_{L^\infty(I; \mathcal L(V; V^*))} \norm{v_m(t)}_V \bigintsss_0^t \norm{v_m(s)}_V \ds   \\
    & \qquad \quad + (T^i \norm{D_i'} + T^{i-1} \norm{D_i}) \bigintsss_0^t \norm{v_m(s)}_V^2 \ds \displaybreak[0]\\
    &\leq  \epsilon \norm{v_m(t)}_V^2 + \left(\frac{T^{2i-1} \norm{D_i}^2}{4\epsilon} + T^i \norm{D_i'} + T^{i-1} \norm{D_i} \right) \bigintsss_0^t \norm{v_m(s)}^2_V \ds,
  \end{align*}
  holds again for some arbitrary $\epsilon>0$. 
  Here, we made use of
  \begin{equation*}
    \int_0^t \tnorm{(R^{i}_0 v_m)(s)} \norm{v_m(s)} \ds \leq T^i \int_0^t \norm{v_m(s)}^2 \ds.
  \end{equation*}
  The fact that $E_i(t), Q(t)\in \mathcal L(V; H)$ allows for
  \begin{equation*}
    - 2 \bigintsss_0^t \dup{Q(s)v_m(s)}{v_m'(s)} \ds \leq \norm{Q}_{L^\infty(I; \mathcal L(V;H))} \bigintsss_0^t \norm{v_m(s)}^2_V + \norm{v_m'(s)}^2_H \ds
  \end{equation*}
  for $Q$ and
  \begin{equation*}
    - 2 \bigintsss_0^t \dup{E_i(s)(R^i_0 v_m)(s)}{v_m'(s)} \ds \leq \norm{E_i}_{L^\infty(I; \mathcal L(V;H))} T^i \bigintsss_0^t \norm{v_m(s)}^2_V + \norm{v_m'(s)}^2_H \ds
  \end{equation*}
  for $E_i$.
  If we insert all these intermediate computations into~\eqref{eq:energy_general1} and choose e.g. $\epsilon = a_0/2$ we arrive at an inequality of the form
  \begin{equation*}
    \norm{v_m(t)}^2_V + \norm{v_m'(t)}^2_H \leq \tilde\Lambda \left(\norm{v_0}_V^2 + \norm{v_1}_H^2 + \norm{f}^2 + \bigintsss_0^t \norm{v_m(s)}^2_V + \norm{v_m'(s)}^2_H \ds \right),
  \end{equation*}
  where $f$ is measured in either the $L^2(I; H)$- or the $H^1(I; V^*)$ norm and $\tilde\Lambda$ is a constant as stated in the assertion. 
  By applying Grönwall's Lemma we conclude
  \begin{equation*}
    \norm{v_m(t)}^2_V + \norm{v_m'(t)}^2_H \leq \tilde\Lambda \left(\norm{v_0}_V^2 + \norm{v_1}_H^2 + \norm{f}^2 \right) \exp(\tilde\Lambda T)
  \end{equation*}
  for almost all $t\in I$.
\end{proof}

As a consequence of the previous result the norms of $v_m\in V_m$ in $L^2(I; V)$ and $H^{1}(I; H)$ are bounded uniformly with respect to $m\in \N$. 
This guarantees existence of a subsequence $v_m$ (also denoted with $m$) converging weakly in $L^2(I; V)\cap H^{1}(I; H)$, whose weak limit we denote by $v$. 
The next step is to show that this weak limit solves the auxiliary problem~\eqref{eq:problem_general_aux_form}.

\begin{lemma}\label{lemma:solvability_general_aux_form}
  For $A\in W^{1,\infty}(I; \mathcal L(V;V^*)), C\in W^{1,\infty}(I; \mathcal L(H))$, $B\in L^{\infty}(I; \mathcal L(H))$, $D_i\in W^{1,\infty}(I; \mathcal L(V;V^*))$, $E_i\in L^\infty(I; \mathcal L(V;H))$ where $i=1, \dots, k$, $f\in L^2(I; H)$ or $f\in H^1(I; V^*)$, $v_0\in V$ and $v_1\in H$ exists $v\in L^2(I; V)\cap H^1(I; H)$ with $(\mathcal Cv')' \in L^2(I; V^*)$ that solves~\eqref{eq:problem_general_aux_form} and also satisfies the energy estimate~\eqref{eq:energy}.
\end{lemma}

\begin{proof}
  We take the weakly converging Galerkin sequence $v_m$ from before and plug them into~\eqref{eq:problem_general_aux_form_a}. 
  Rearranging this equation into
  \begin{equation*}
    (\mathcal{C}v_m')'=f - \mathcal{B}v_m' - (\mathcal{A} + \mathcal{Q})v_m - \sum_{i=1}^k(\mathcal{D}_i + \mathcal{E}_i)R_0^i v_m
  \end{equation*}
  shows that, according to the assumptions made beforehand, the right-hand side is in $L^2(I; V^*)$, particularly, this shows the right-hand side to be \emph{weak-weak continuous} in $v_m$. 
  Hence, $v_m''$ converges weakly in $L^2(I; V^*)$, i.e.,
  \begin{equation*}
    (\mathcal{C}v_m')'\rightharpoonup f - \mathcal{B}v' - (\mathcal{A} + \mathcal{Q})v - \sum_{i=1}^k(\mathcal{D}_i + \mathcal{E}_i)R_0^i v.
  \end{equation*}
  Trivially, this implies $(Cv_m')'\rightharpoonup (Cv')'$ in $L^2(I; V^*)$, too, thus proving to solve~\eqref{eq:problem_general_aux_form_a}. 
  We are left to show that $v$ satisfies the initial conditions~\eqref{eq:problem_general_aux_form_b}. 
  By design of~\eqref{eq:3.3b}, we have
  \begin{equation}\label{eq:3.9}
    v_m(0) \to v_0\text{ in }H\text{ for }m\to\infty, \quad (\mathcal{C}v_m')(0) \to v_1\text{ in }H\text{ for }m\to\infty.
  \end{equation}
  Since $W^{1,2}(I; H)\subset C^0(I; H)$ as well as $g\mapsto (g(0),\psi)_H \in C^0(I; H)^*$ for all $\psi\in H$ hold, respectively, we infer
  \begin{equation*}
    (v_m(0),\psi)_H \to (v(0),\psi)_H,\quad \text{ for all } \psi\in H
  \end{equation*}
  showing $v_m(0)\rightharpoonup v(0)=v_0$ in $H$. 
  Together with~\eqref{eq:3.9} we obtain strong convergence in $H$, such that $v(0)=v_0$ holds (even in $V$). 
  An analogous argument then shows $(\mathcal{C}v')(0)=v_1$ in $V^*$ when keeping $C(0)\in\mathcal{L}(H)$ in mind. 
  This completes the proof.
\end{proof}

We are now in a position to apply this existence theorem to our auxiliary problem~\eqref{eq:problem_general_aux}.

\begin{corollary}\label{corollary:aux_solvable}
    Let $k\in\N$. 
    Suppose $A\in W^{k+1,\infty}(I; \mathcal L(V;V^*))$, $Q\in W^{k,\infty}(I; \mathcal L(V;H))$, $B\in W^{k,\infty}(I; \mathcal L(H))$,
    $C\in W^{k+1,\infty}(I; \mathcal L(H))$, $f\in H^k(I; H)$ or $f\in H^{k+1}(I; V^*)$, $u_j\in V$ for $j=0,\dots, k$ and $u_{k+1}\in H$.
    Then there exists $v\in L^2(I; V)\cap H^1(I; H)$ with $(\mathcal Cv')' \in L^2(I; V^*)$ that solves~\eqref{eq:problem_general_aux} and also satisfies the energy estimate
    \begin{equation}\label{eq:aux_energy}
      \esssup_{t\in I} \left(\norm{v(t)}^2_V + \norm{v'(t)}^2_H\right) \leq \Lambda \left(\sum_{j=0}^k \norm{u_j}_V^2 + \norm{u_{k+1}}_H^2 + \tnorm{f^{(k)}}^2 \right)
    \end{equation}
    where $f^{(k)}$ is measured in either the $L^2(I; H)$- or the $H^{1}(I; V^*)$ norm and $\Lambda=\Lambda(k)$ is a constant depending continuously on $1/{c_0}$, $1/{a_0}$, $T$ and the operators $A, B, C, Q$, measured in the spaces above.
\end{corollary}

\begin{proof}
    To obtain~\eqref{eq:problem_general_aux_form} from~\eqref{eq:problem_general_aux}, we set
    \begin{align*}
      \tilde B &= B+C', &\tilde Q &= Q + kB' + \frac{k(k+1)}2 C'',\\
      D_j &= {k \choose j} A^{(j)}, &E_j &= {k+1 \choose j+2} C^{(j+2)} + {k \choose j+1} B^{(j+1)} + {k \choose j} Q^{(j)},
    \end{align*}
    for $j=1,\dots,k$. 
    The operators $A$ and $C$ are left as-is, and the initial values are $u_k$ and $u_{k+1}$ as defined in~\eqref{eq:def_initial_values}. 
    The right-hand side $\tilde f$ consists of $f^{(k)}$ and those parts of
    \begin{equation*}
      \left(\compose_{l=k-j}^{k-1} R_{u_l} v\right)(t) = (R_0^j v)(t) + \sum_{l=k-j}^{k-1} (R_0^{k-1-l} u_l)(t) = (R_0^j v)(t) + \sum_{l=k-j}^{k-1} u_l \frac{t^{k-1-l}}{(k-1-l)!}
    \end{equation*}
    that are independent of $v$, i.e.,
    \begin{equation*}
      \tilde f(t) = f^{(k)}(t) - \sum_{j=1}^k (D_j(t)+E_j(t)) \sum_{l=k-j}^{k-1} u_l \frac{t^{k-1-l}}{(k-1-l)!}.
    \end{equation*}
    The assertion follows by utilizing \autoref{lemma:solvability_general_aux_form}.
\end{proof}

\section{Proof of the main theorem}\label{section:main_proof}

We have seen that solutions to the auxiliary problems~\eqref{eq:problem_general_aux} exist, but have yet to show that they are linked to~\eqref{eq:problem_general}. 
As they were obtained using formal differentiation, it is somewhat intuitive to see if functions obtained by integrating these solutions with respect to time over $(0,t)$ indeed solve the previous problem. 
The following lemma resolves the situation.

\begin{lemma}\label{lemma:inductive}
  Suppose $v$ is a solution to problem~\eqref{eq:problem_general_aux} with $k\in\N$ and $u_{k-1}\in V$. 
  Then $R_{u_{k-1}} v$ solves problem~\eqref{eq:problem_general_aux} with $k$ replaced by $k-1$.
\end{lemma}

\begin{proof}
  We set $w=R_{u_{k-1}}v$. 
  Clearly $w$ possesses the correct initial values since $w(0)=u_{k-1}$ and $w'(0)=v(0)=u_k$. 
  By a straightforward (albeit lengthy) computation we show that
  \begin{multline}\label{eq:lemma_general_inductive1}
    (\mathcal{C}w')' = f^{(k-1)} - ((k-1)\mathcal C'+\mathcal B)w'\\
      - \sum_{j=0}^{k-1} \left[{k-1 \choose j} (\mathcal A^{(j)}+\mathcal Q^{(j)}) + {k-1 \choose j+1} \mathcal B^{(j+1)} + {k \choose j+2}\mathcal C^{(j+2)}\right] \compose_{l=k-1-j}^{k-2} R_{u_l} w
  \end{multline}
  holds true in the $L^2(I; V^*)$ sense. 
  We start by employing the fundamental theorem of calculus to the sum in the preceding equation and compute
  \begin{multline}
    \sum_{j=0}^{k-1} \left[{k-1 \choose j}  (\mathcal A^{(j)}+\mathcal Q^{(j)}) + {k-1 \choose j+1}  \mathcal B^{(j+1)} + {k \choose j+2} \mathcal C^{(j+2)}\right] \left(\compose_{l=k-1-j}^{k-1} R_{u_l} v\right)(t) \\
    = \sum_{j=0}^{k-1} \left[{k-1 \choose j}  (A^{(j)}(0)+Q^{(j)}(0)) + {k-1 \choose j+1}  B^{(j+1)}(0) + {k \choose j+2} C^{(j+2)}(0)\right] u_{k-1-j} \\
    + \bigintsss_0^t \left(\sum_{j=0}^{k-1} \left[{k-1 \choose j} (\mathcal A^{(j)}+\mathcal Q^{(j)}) + {k-1 \choose j+1}  \mathcal B^{(j+1)} \right.\right.\\
     \left. \left. + {k \choose j+2} \mathcal C^{(j+2)}\right] \left(\compose_{l=k-1-j}^{k-1} R_{u_l} v\right)(s)\right)' \ds.
  \end{multline}
  Turning to~\eqref{eq:def_initial_values} shows that the expression before the integral can be rewritten using $C(0)u_{k+1}$. 
  Using~\autoref{lemma:product rule} we are able to recast the integral expression, too, and end up with
  \begin{multline}
    \sum_{j=0}^{k-1} \left[{k-1 \choose j}  (\mathcal A^{(j)}+\mathcal Q^{(j)}) + {k-1 \choose j+1}  \mathcal B^{(j+1)} + {k \choose j+2} \mathcal C^{(j+2)}\right] \left(\compose_{l=k-1-j}^{k-1} R_{u_l} v\right)(t) \\
      =-C(0)u_{k+1} + f^{(k-1)}(0) - (kC'(0)+B(0))u_k + \bigintsss_0^t \sum_{j=0}^{k-1} \left[{k-1 \choose j}  (\mathcal A^{(j+1)} \right.\\
       \left. + \mathcal Q^{(j+1)}) + {k-1 \choose j+1}  \mathcal B^{(j+2)} + {k \choose j+2} \mathcal C^{(j+3)}\right] \left(\compose_{l=k-1-j}^{k-1} R_{u_l} v\right)(s) \ds \\
      + \bigintsss_0^t \sum_{j=0}^{k-1} \left[{k-1 \choose j}  (\mathcal A^{(j)}+\mathcal Q^{(j)}) + {k-1 \choose j+1}  \mathcal B^{(j+1)} \right. \\
        \left. + {k \choose j+2} \mathcal C^{(j+2)}\right] \left(\compose_{l=k-j}^{k-1} R_{u_l} v\right)(s) \ds.
  \end{multline}
  Shifting the index of the sum in the first integral on the left-hand side to $j+1$, using an identity for binomial coefficients, and adding a zero then shows
  \begin{multline*}
    \sum_{j=0}^{k-1} \left[{k-1 \choose j}  (\mathcal A^{(j)}+\mathcal Q^{(j)}) + {k-1 \choose j+1}  \mathcal B^{(j+1)} + {k \choose j+2} \mathcal C^{(j+2)}\right] \left(\compose_{l=k-1-j}^{k-1} R_{u_l} v\right)(t) \\
    = -C(0)u_{k+1} + f^{(k-1)}(0) - (kC'(0)+B(0))u_k - \bigintsss_0^t (k\mathcal C''+\mathcal B')(v)(s) \ds \hspace{2.5cm}\\
    + \bigintsss_0^t \sum_{j=0}^{k} \left[{k \choose j}  (\mathcal A^{(j)}+\mathcal Q^{(j)}) + {k \choose j+1}  \mathcal B^{(j+1)} + {k+1 \choose j+2} \mathcal C^{(j+2)}\right] \left(\compose_{l=k-j}^{k-1} R_{u_l} v\right)(s) \ds,
  \end{multline*}
  if we combine the previous results. 
  Adding this to the remaining parts of~\eqref{eq:lemma_general_inductive1} lets us conclude
  \begin{multline*}
    \Big[f^{(k-1)}-((k-1)\mathcal C'+\mathcal B)w' - \sum_{j=0}^{k-1} \left[{k-1 \choose j} (\mathcal A^{(j)} \hfill\right. \\
        \quad\left.+\mathcal Q^{(j)}) + {k-1 \choose j+1} \mathcal B^{(j+1)} + {k \choose j+2}\mathcal C^{(j+2)}\right] \compose_{l=k-1-j}^{k-2} R_{u_l} w \Big](t) \hfill \\
    \qquad = C(0)u_{k+1} + C'(0)u_k + \bigintsss_0^t f^{(k)}(s) + (\mathcal C' v)'(s) - ((k\mathcal C' + \mathcal B) v)'(s) + ((k\mathcal C''+\mathcal B')v)(s) \ds \hfill \\
    \qquad\quad - \bigintsss_0^t \sum_{j=0}^{k} \left[{k \choose j}  (A^{(j)}(s)+Q^{(j)}(s)) + {k \choose j+1}  B^{(j+1)}(s) \right. \hfill \\
      \quad\left.+ {k+1 \choose j+2} C^{(j+2)}(s)\right] \left(\compose_{l=k-j}^{k-1} R_{u_l} v\right)(s) \ds \\
    \qquad = C(0)u_{k+1} + C'(0)u_k + \bigintsss_0^t (\mathcal C v')'(s) + (\mathcal C' v)'(s) \ds \hfill \\
    \qquad = C(0)u_{k+1} + C'(0)u_k + \bigintsss_0^t (\mathcal C v)''(s) \ds = (\mathcal Cv)'(t) = (\mathcal C w')'(t).\qedhere
  \end{multline*}
\end{proof}

Now to prove our main theorem we can use the results from before to make a compact induction argument.

\begin{theorem:main}
  Let $k\in\N\cup \{0\}$, $l=\max \{1,k\}$ and suppose that $A\in W^{k+1,\infty}(I; \mathcal L(V;V^*))$, $Q\in W^{l,\infty}(I; \mathcal L(V;H))$, $B\in W^{l,\infty}(I; \mathcal L(H))$, $C\in W^{k+1,\infty}(I; \mathcal L(H))$, as well as
  $f\in H^k(I; H)$ or $f\in H^{k+1}(I; V^*)$, and $u_j\in V$ for $j=0,\dots, k$ and $u_{k+1}\in H$.
  Then the unique solution $u$ of problem~\eqref{eq:problem_general} lies in $H^k(I; V)\cap H^{k+1}(I; H)$ with $(\mathcal C u^{(k+1)})' \in L^2(I; V^*)$ and satisfies the energy estimate
  \begin{equation}\tag{\ref{eq:theorem_main_energy}}
  \norm{u}_{W^{k,\infty}(I; V)}^2 + \tnorm{u^{(k+1)}}_{L^{\infty}(I; H)}^2 \leq \Lambda \left(\sum_{j=0}^k \norm{u_j}_V^2 + \norm{u_{k+1}}_H^2 + \norm{f}^2 \right),
  \end{equation}
  where $f$ is measured in either the $H^k(I; H)$- or the $H^{k+1}(I; V^*)$ norm and $\Lambda=\Lambda(k)$ is a constant depending continuously on $1/{c_0}$, $1/{a_0}$, $T$ and the operators $A, B, C, Q$, measured in the spaces above.
\end{theorem:main}

\begin{proof}
  Induction over $k\in\N\cup \{0\}$, where the hypotheses is that the assertion as given in the theorem holds for $k$ and that $u^{(k)}$ is the unique solution of~\eqref{eq:problem_general_aux}. \\[2ex]
  The case $k=0$ is covered by \autoref{theorem:well_posedness}. 
  We assume that the hypotheses holds for $k-1$ and the requirements for $k$ are met.
  Using \autoref{corollary:aux_solvable} we conclude that a solution $v\in L^2(I; V)\cap H^1(I; H)$ with $(\mathcal C v')' \in L^2(I; V^*)$ to the auxiliary problem~\eqref{eq:problem_general_aux}, which also fulfills the estimates from equation~\eqref{eq:aux_energy}, exists.
  Due to \autoref{lemma:inductive} we know that $R_{u_{k-1}} v$ satisfies~\eqref{eq:problem_general_aux} for $k-1$, which is uniquely solved by $u^{(k-1)}$.
  Therefore $v = (R_{u_{k-1}} v)'=u^{(k)}$ (which is unique) and $u^{(k)}$ satisfies the estimate~\eqref{eq:aux_energy}, which we add to the energy estimates for the case $k-1$ to obtain~\eqref{eq:theorem_main_energy}.
\end{proof}

\begin{remark}
  Note that in the case of $l=k=0$ the existence of a solution to~\eqref{eq:problem_general} still holds, but not its uniqueness.
\end{remark}

\subsubsection*{Acknowledgements}

The authors acknowledge funding of T.\ Gerken by the Deutsche Forschungsgemeinschaft (DFG, German Research Foundation) – Project number 281474342/GRK2224/1.

\addcontentsline{toc}{section}{References}
\printbibliography%

\end{document}